\documentclass[a4paper,oneside,reqno]{amsart}
\usepackage[utf8]{inputenc}

\usepackage{amssymb,amsmath}
\usepackage{thmtools}
\usepackage{mathtools}
\usepackage{enumitem}
\usepackage{graphicx}
\usepackage{hyperref}
\usepackage[capitalize]{cleveref}
\usepackage{stmaryrd} 
\usepackage{tikz}

\usepackage{geometry}
\declaretheorem[numberwithin=section]{theorem}
\declaretheorem[sibling=theorem]{lemma}
\declaretheorem[sibling=theorem]{fact}
\declaretheorem[sibling=theorem]{claim}

\declaretheorem[sibling=theorem]{conjecture}
\declaretheorem[sibling=theorem]{corollary}
\declaretheorem[sibling=theorem]{proposition}
\declaretheorem[sibling=theorem,style=definition]{definition}

\crefname{fact}{Fact}{Facts}
\crefname{claim}{Claim}{Claims}
\crefname{hypothesis}{Hypothesis}{Hypotheses}
\crefname{conjecture}{Conjecture}{Conjectures}
\Crefname{conjecture}{Conjecture}{Conjectures}

\makeatletter
\renewcommand{\theHtheorem}{\thesection.\arabic{theorem}}
\newcommand{\FixSiblingH}[1]{%
  \expandafter\renewcommand\csname theH#1\endcsname{\theHtheorem}}
\makeatother

\FixSiblingH{lemma}

\FixSiblingH{fact}
\FixSiblingH{claim}
\FixSiblingH{hypothesis}
\FixSiblingH{conjecture}
\FixSiblingH{corollary}
\FixSiblingH{proposition}
\FixSiblingH{definition}
\FixSiblingH{remark}

\newcommand{\indsetgraph}{
  \begin{tikzpicture}[baseline={(0,-0.1)},
    every loop/.style={},
    vertex/.style={circle,fill=black,inner sep=1.2pt},
    edge/.style={line width=0.5pt}]
    \clip (-0.05,-0.15) rectangle (0.8,0.15);
    \node[vertex] (v1) at (0,0) {};
    \node[vertex] (v2) at (0.5,0) {};
    \draw[edge] (v1) -- (v2);
    \path[edge] (v2) edge[in=45, out=-45, min distance=12] (v2);
  \end{tikzpicture}
}

\DeclareMathOperator{\ex}{\mathrm{ex}}
\DeclareMathOperator{\Ex}{\mathbb{E}}
\DeclareMathOperator{\Hom}{\mathrm{Hom}}
\DeclareMathOperator{\DKulLei}{D_{\mathrm{KL}}}

\renewcommand{\Pr}{\mathbb{P}}
\renewcommand{\emptyset}{\varnothing}

\newcommand{\DKL}[2]{\DKulLei({#1} \, \| \, {#2})}
\newcommand{\dKL}{d_{\mathrm{KL}}}
\newcommand{\dTV}{d_{\mathrm{TV}}}
\newcommand{\Ber}{\mathrm{Ber}}

\newcommand{\br}[1]{\llbracket{#1}\rrbracket}

\title{Entropy methods in combinatorics}
\author{Wojciech Samotij}
\address{School of Mathematical Sciences, Tel Aviv University, Tel Aviv, Israel}
\email{samotij@tauex.tau.ac.il}
\date{}

\begin{document}

\begin{abstract}
  Even though entropy methods have been used in combinatorics for at least five decades, only in recent years has their use really proliferated.
  There are now tens, if not hundreds, of combinatorial papers that crucially rely on the notion of entropy and exploit the various powerful identities and inequalities relating entropies.
  In this short survey article, we give a selective overview of these works and discuss several of them in more detail, outlining some of the key ideas.
\end{abstract}

\maketitle

\section{Introduction.}
The notion of \emph{entropy} was introduced by Claude Shannon in his landmark 1948 paper~\cite{Sha48} that laid rigorous foundations for the field of information theory.
Shannon's entropy quantifies the amount of uncertainty that is associated with a discrete random variable.
The fundamental property of entropy that makes it relevant to enumeration problems is that the entropy of a random variable that is uniformly distributed on a~finite set $\mathcal{X}$ is the logarithm of the cardinality of $\mathcal{X}$.
This fact allows one to turn counting problems into problems of estimating entropy.
The advantage of this seemingly circuitous approach is that there are several powerful inequalities involving entropies that do not seem to have any direct counting analogues.
The origins of the idea of using entropy in combinatorics are hard to trace, but it has surely been around for at least five decades (see, e.g., \cite{ChuGraFraShe86} or \cite[Chapter~2]{CsiKor11}).
Once a niche tool, since the turn of the century entropy has become a~widely-applied technique.
The aim of this survey is to present some of the most influential applications of entropy methods in combinatorics and highlight the underlying ideas.

\subsection{Entropy.}
\label{sec:entropy}

There are several different ways to introduce and motivate the notion of entropy, some of which are discussed in detail by Rényi~\cite{Ren65}.
The author is particularly fond of the following axiomatic approach he first encountered in the excellent lecture notes of Galvin~\cite{Gal14}.
We would like to define a function $I$ that measures the information one learns from the occurrence of a certain subset of outcomes of a random experiment;
one then defines the entropy of a discrete random variable $X$ as the average amount of information associated with the outcome of $X$.
We postulate that $I$ should: (i) take only nonnegative values and depend only on the probability of the said event; (ii) be decreasing in this probability (less likely events should carry more information); and (iii) assign to the intersection of two independent events the sum of the amounts of information associated with the two events.
In other words, we are looking for a function $I \colon (0,1] \to [0,\infty)$ that is decreasing and satisfies $I(pq) = I(p)+I(q)$ for all $p, q \in (0,1]$.
Since one can show that the only functions that meet this specification are logarithmic functions, that is, $I(p) = - \log_b p$ for some $b > 1$, this leads to the following definition.
\begin{definition}
  The entropy of a random variable $X$ taking values in a finite set $\mathcal{X}$ is
  \[
    H(X) \coloneqq  - \sum_{x \in \mathcal{X}} \Pr(X = x) \cdot \log\Pr(X = x).\footnote{Throughout, $\log$ is the natural logarithm and we follow the convention that $0 \log 0 = 0$.}
  \]
\end{definition}
The following simple fact is a generalisation of the fundamental property of entropy mentioned in the introduction.
\begin{fact}
  \label{fact:entropy-uniform}
  If a random variable $X$ takes values in a finite set $\mathcal{X}$, then
  \[
    H(X) \le \log |\mathcal{X}|.
  \]
  Moreover, equality holds if and only if $X$ is uniform on $\mathcal{X}$.
\end{fact}

\subsection{Conditional entropy.}
\label{sec:conditional-entropy}

A fundamental notion that makes the entropy method so powerful is that of \emph{conditional entropy}.
Informally speaking, the conditional entropy of a random variable $X$ given another random variable $Y$ (defined on the same probability space) is the expected amount of information learned from the outcome of $X$ assuming that the outcome of $Y$ is already known.
More precisely, for every $y$ such that $\Pr(Y = y)$ is nonzero, define the conditioned random variable $X^y$ by $\Pr(X^y = x) \coloneqq \Pr(X = x \mid Y = y)$.
The conditional entropy of $X$ given~$Y$ is then the expectation of the (random) entropy $H(X^y)$, where $y$ is sampled according to $Y$.
\begin{definition}
  Suppose that $X$ and $Y$ are random variables taking values in finite sets $\mathcal{X}$ and $\mathcal{Y}$, respectively.
  The conditional entropy of $X$ given $Y$ is
  \[
    H(X \mid Y) \coloneqq - \sum_{y \in \mathcal{Y}} \Pr(Y = y) \sum_{x \in \mathcal{X}} \Pr(X = x \mid Y = y) \log \Pr(X = x \mid Y = y).
  \]
\end{definition}
It is easily verified that $H(X \mid Y) = H(X, Y) - H(Y)$, where we wrote $H(X,Y)$ in place of $H((X,Y))$ to denote the entropy of the random vector $(X,Y)$.
The straightforward generalisation of this identity to $n$-dimensional vectors is the so-called \emph{chain rule} for entropies.
\begin{fact}[Chain rule]
  \label{fact:chain-rule}
  For every sequence $X_1, \dotsc, X_n$ of discrete random variables,
  \[
    H(X_1, \dotsc, X_n) = \sum_{i = 1}^n H(X_i \mid X_1, \dotsc, X_{i-1}).
  \]
\end{fact}
The next key property formalises the intuitive statement that conditioning on more information can only decrease conditional entropy.
\begin{fact}
  \label{fact:monotonicity-of-conditional-entropy}
  For every triple $X, Y, Z$ of discrete random variables,
  \[
    H(X \mid Y, Z) \le H(X \mid Y) \le H(X).
  \]
  Moreover, the second inequality holds with equality if and only if $X$ and $Y$ are independent while the first inequality holds with equality if and only if $X$ and $Z$ are conditionally independent given $Y$.\footnote{This means that, for every $y$ with $\Pr(Y = y) > 0$, the conditioned random variables $X^y$ and $Z^y$ are independent.}
\end{fact}
An immediate consequence of \cref{fact:chain-rule,fact:monotonicity-of-conditional-entropy} is the following subadditivity property.
\begin{fact}[Subadditivity]
  \label{fact:entropy-subadditivity}
  For every sequence $X_1, \dotsc, X_n$ of discrete random variables,
  \[
    H(X_1, \dotsc, X_n) \le \sum_{i=1}^n H(X_i).
  \]
  Moreover, equality holds if and only if $X_1, \dotsc, X_n$ are independent.
\end{fact}
Similarly, one can derive the following more general, conditional version of the subadditivity property.
\begin{fact}
  \label{fact:conditional-entropy-subadditivity}
  For every sequence $X_1, \dotsc, X_n, Y$ of discrete random variables,
  \[
    H(X_1, \dotsc, X_n \mid Y) \le \sum_{i=1}^n H(X_i \mid Y)
  \]
  and equality holds if and only if $X_1, \dotsc, X_n$ are conditionally independent given $Y$.
\end{fact}

\subsection{Binary entropy.}
\label{sec:binary-entropy}

In many applications of the entropy method, one often considers the entropy of two-valued random variables.
It is therefore customary to introduce the \emph{binary entropy} function, which assigns to a given parameter $p \in [0,1]$ the entropy of a random variable that takes only two values, with probabilities $p$ and $1-p$, respectively (e.g., the Bernoulli random variable with success probability $p$, denoted by $\Ber(p)$).
\begin{definition}
  The binary entropy is the function $h \colon [0,1] \to [0, \log 2]$ defined by
  \[
    h(p) \coloneqq H(\Ber(p)) = -p\log p - (1-p)\log(1-p).
  \]
\end{definition}
The binary entropy function arises naturally in the following classical estimate, whose elegant, entropy-based proof is perhaps the archetypal application of the entropy method.
\begin{proposition}
  For all integers $k$ and $n$ satisfying $0 \le k \le n/2$,
  \[
    \sum_{j = 0}^{k} \binom{n}{j} \le \exp(n \cdot h(k/n)).
  \]
\end{proposition}
\begin{proof}
  Fix integers $k$ and $n$ satisfying $0 \le k \le n$ and let $\Sigma$ denote the left-hand side of the asserted inequality.
  Denote by $(X_1, \dotsc, X_n) \in \{0,1\}^n$ the characteristic function of the uniformly random subset of $\br{n} \coloneqq \{1, \dotsc, n\}$ of cardinality at most $k$.
  By \cref{fact:entropy-uniform,fact:entropy-subadditivity} and the definition of $h$,
  \[
    \log \Sigma = H(X_1, \dotsc, X_n) \le \sum_{i=1}^n H(X_i) = \sum_{i=1}^n h(\Ex[X_i]).
  \]
  By symmetry, for each $i$,
  \[
    n \cdot \Ex[X_i] = \sum_{j=1}^n \Ex[X_j] = \Ex\left[\sum_{j=1}^n X_j\right] \le k,
  \]
  which implies that
  \[
    h(\Ex[X_i]) \le \sup\{h(p) : 0 \le p \le k/n\}.
  \]
  Assume now that $k \le n/2$.
  The claimed inequality follows after we observe that $h'(p) = \log ((1-p)/p)$ for all $p \in (0,1)$, which means that $h$ is increasing on $[0,1/2]$.
\end{proof}

\subsection{Relative entropy.}

Entropy is a special case of the much more general concept of relative entropy, first introduced by Kullback and Leibler~\cite{KulLei51}.
Given two probability distributions $P$ and $Q$ defined on the same finite set $\mathcal{X}$ and such that $Q$ is absolutely continuous\footnote{Recall that this means that $Q(x) = 0$ for every $x \in \mathcal{X}$ such that $P(x) = 0$} with respect to $P$ (which we will denote from now on by $Q \ll P$) we define the \emph{relative entropy} of $Q$ from $P$ (also called the \emph{Kullback--Leibler divergence} of $Q$ from $P$) by
\[
  \DKL{Q}{P} \coloneqq \sum_{x \in \mathcal{X}} Q(x) \log \frac{Q(x)}{P(x)}.
\]
In the sequel, we shall often identify an $\mathcal{X}$-valued random variable $X$ with its probability distribution, denoted by $\mathcal{L}(X)$.
In particular, $\DKL{X}{Y}$ is nothing else than  $\DKL{\mathcal{L}(X)}{\mathcal{L}(Y)}$.

An information-theoretic motivation for the above definition is based on the notion of cross-entropy.
One of the main findings of the celebrated work of Shannon~\cite{Sha48} is that the ideal lossless coding scheme optimised for outcomes of a discrete random variable $X$ uses $H(X)$ units of information on average.
The \emph{cross-entropy} of $X$ relative to another random variable $Y$ taking the same set of values, which we will denote here by $H(X ; Y)$, is the expected number of units of information used by the coding scheme optimised for $Y$ while encoding outcomes of $X$; in other words,
\[
  H(X ; Y) \coloneqq - \sum_{x \in \mathcal{X}} \Pr(X = x) \log\Pr(Y = x).
\]
Now, the relative entropy of $X$ from $Y$ is the difference between the cross-entropy $H(X;Y)$ and the real entropy of $X$, that is,
\begin{equation}
  \label{eq:relative-cross-entropy}
  \DKL{X}{Y} = H(X;Y) - H(X).
\end{equation}
In other words, $\DKL{X}{Y}$ is the average excess amount of information wasted by using the coding scheme optimised for $Y$ in order to encode the outcomes of $X$.
Finally, let $U_{\mathcal{X}}$ denote the uniformly chosen random element of a finite set $\mathcal{X}$.
It is easily checked that every $\mathcal{X}$-valued random variable $X$ satisfies $H(X ; U_{\mathcal{X}}) = \log |\mathcal{X}|$, and thus \eqref{eq:relative-cross-entropy} implies that
\begin{equation}
  \label{eq:entropy-minus-relative-entropy}
  H(X) = \log|\mathcal{X}| - \DKL{X}{U_{\mathcal{X}}},
\end{equation}
which explains why relative entropy generalises the notion of entropy.

The key property of relative entropy is that it is always nonnegative.

\begin{fact}
  \label{fact:relative-entropy-nonnegative}
  For every pair $P, Q$ of distributions (on a finite set $\mathcal{X}$) that satisfy $Q \ll P$, we have
  \[
    \DKL{Q}{P} \ge 0.
  \]
  Moreover, equality holds if and only if $P = Q$.
\end{fact}
\begin{proof}
  Since $\log y \ge 1 - 1/y$ for all $y \in (0,\infty)$, and the inequality is strict unless $y = 1$, we have
  \[
    \DKL{Q}{P} = \sum_{\substack{x \in \mathcal{X} \\ Q(x) > 0}} Q(x) \log \frac{Q(x)}{P(x)} \ge \sum_{\substack{x \in \mathcal{X} \\ Q(x) > 0}} Q(x) \left(1 - \frac{P(x)}{Q(x)}\right) \ge \sum_{x \in \mathcal{X}}Q(x) - \sum_{x \in \mathcal{X}}P(x) = 0
  \]
  and equality holds throughout if and only if $P(x) = Q(x)$ for all $x \in \mathcal{X}$.
\end{proof}

Observe that \cref{fact:relative-entropy-nonnegative}, together with \eqref{eq:entropy-minus-relative-entropy}, implies \cref{fact:entropy-uniform}.
Interestingly, \cref{fact:monotonicity-of-conditional-entropy,fact:entropy-subadditivity,fact:conditional-entropy-subadditivity} can also be deduced from nonnegativity of relative entropy between carefully chosen pairs of distributions.
For example, \cref{fact:entropy-subadditivity} is an immediate consequence of the following identity, where we write $Y \otimes Z$ for the random vector whose coordinates are independent and have the same distributions as $Y$ and $Z$, respectively:
\begin{equation}
  \label{eq:subadditivity-via-DKL}
  H(X_1) + \dotsb + H(X_n) - H(X_1, \dotsb, X_n) = \DKL{(X_1, \dotsc, X_n)}{X_1 \otimes \dotsb \otimes X_n}.
\end{equation}
We state one more useful consequence of~\cref{fact:relative-entropy-nonnegative}, commonly known as the \emph{data processing inequality}.
Roughly speaking, it states that applying the same deterministic transformation to a pair of random variables cannot increase their relative entropy.
\begin{fact}
  \label{fact:data-processing}
  Suppose that random variables $X$ and $Y$ take values in the same finite set $\mathcal{X}$ and satisfy $\mathcal{L}(X) \ll \mathcal{L}(Y)$.
  For every set $\mathcal{T}$ and every function $T \colon \mathcal{X} \to \mathcal{T}$, we have
  \[
    \DKL{T(X)}{T(Y)} \le \DKL{X}{Y}.
  \]
\end{fact}
\begin{proof}
  The claimed inequality follows as
  \begin{equation}
    \label{eq:data-processing}
    \DKL{X}{Y} - \DKL{T(X)}{T(Y)} = \sum_{x \in \mathcal{X}} \Pr(X = x) \log\frac{\Pr(X = x)\Pr(T(Y)=T(x))}{\Pr(Y=x)\Pr(T(X)=T(x))}
  \end{equation}
  and it is easily checked that the function $P \colon \mathcal{X} \to [0,\infty)$ defined by
  \[
    P(x) \coloneqq \frac{\Pr(T(X)=T(x))}{\Pr(T(Y)=T(x))} \cdot \Pr(Y = x)
  \]
  is a probability distribution satisfying $\mathcal{L}(X) \ll P$ and thus \eqref{eq:data-processing} is equal to the nonnegative quantity $\DKL{X}{P}$.
\end{proof}

\subsection{Organisation.}
\label{sec:organisation}

In the remainder of the article, we attempt to survey the many uses of entropy methods in extremal and probabilistic combinatorics.
We organise the surveyed papers around common themes and conceptual threads, aiming to maintain a chronological order: the chain rule with randomised order of conditioning (\Cref{sec:randomised-chain-rule}), Shearer's inequality (\Cref{sec:Shearer-inequality}), constructing random graph homomorphisms with large entropy (\Cref{sec:Sidorenko}), Pinsker's inequality and the interplay between entropy and independence (\Cref{sec:entropy-independence}), the recent breakthrough on the union-closed sets conjecture (\Cref{sec:union-closed}), and the very recent entropy-based approach to Turán-type problems (\Cref{sec:Turan-via-entropy}).

\section{Randomised chain rule.}
\label{sec:randomised-chain-rule}

One of the most striking and influential applications of entropy in combinatorics is the beautiful proof of the following conjecture of Minc~\cite{Min63} due to Jaikumar Radhakrishnan~\cite{Rad97}.

\begin{conjecture}[Minc]
  Suppose that $M$ is an $n \times n$ matrix with all entries in $\{0,1\}$ whose row sums are $d_1, \dotsc, d_n$.
  Then, the permanent of $M$ is at most $\prod_{i=1}^n (d_i!)^{1/d_i}$.
\end{conjecture}

Minc's conjecture was proved by Brégman~\cite{Bre73} and, a few years later, Schrijver~\cite{Sch78} found a shorter proof (whose randomised version is presented in~\cite[Chapter~2]{AloSpe16}).
However, it is the beautiful entropy-based argument due to Radhakrishnan that is considered the `book' proof of this result (see~\cite[Chapter~37]{AigZie18}).
The core idea of Radhakrishnan's argument, which we will term here the \emph{randomised chain rule}, proved to be extremely flexible and it has been adapted by a great number of later works.
Notably, Cuckler and Kahn~\cite{CucKah09-UB} used this method to prove an upper bound on the number of perfect matchings and Hamilton cycles in a graph, which they also matched from below~\cite{CucKah09-LB};
their results have recently been generalised to hypergraphs~\cite{DaiDivKel,Kah23,KwaSafWan}.
In another notable application of the randomised chain rule, Linial and Luria prove upper bounds on the number of Steiner triple systems~\cite{LinLur13} and high-dimensional permutations~\cite{LinLur14};
the former bound was later extended to $(n,q,r,\lambda)$-designs, for arbitrary fixed $q$, $r$, $\lambda$ and large $n$, by Keevash~\cite{Kee18}, whose remarkable work~\cite{Kee} supplies matching lower bounds.
Further applications of the randomised chain rule include~\cite{BoyDasSza20,ChrDraGirHurMicMuy,CutRad11,Kwa20,Lur,Sim23,HofPenPol22}.

Consider the problem of estimating the entropy of a random vector $X = (X_i)_{i \in I} \in A^I$, where $A$ and $I$ are two finite sets.
Given a linear order $\prec$ on the set of indices $I$ and writing $X_{\prec i}$ as a shorthand for the vector $(X_j)_{j \prec i}$, we may use the chain rule (\cref{fact:chain-rule}) to conclude that
\begin{equation}
  \label{eq:random-chain-rule}
  H(X) = \sum_{i \in I} H(X_i \mid X_{\prec i}).
\end{equation}
Further, given an $x \in A^I$ such that $\Pr(X=x)$ is nonzero, write $X_i^{\prec,x}$ to denote the variable $X_i$ conditioned on the event that $X_{\prec i} = x_{\prec i}$.
By the definition of conditional entropy, we have
\begin{equation}
  \label{eq:conditional-entropy-random-chain-rule}
  H(X_i \mid X_{\prec i}) = \Ex[H(X_i^{\prec,x})],
\end{equation}
where the expectation averages over $x$ that is sampled according to $X$.
The key realisation that underlies the argument of Radhakrishnan is that both \eqref{eq:random-chain-rule} and \eqref{eq:conditional-entropy-random-chain-rule} hold also for a random ordering $\prec$ and therefore
\begin{equation}
  \label{eq:Ex-random-chain-rule}
  H(X) = \sum_{i \in I} \Ex[H(X_i^{\prec,x})],
\end{equation}
where the expectation averages over both $x$ and the random ordering $\prec$.
Crucially, even though we obtained~\eqref{eq:Ex-random-chain-rule} by first averaging over $x$ (randomness inherent in the problem) and then over $\prec$ (external randomness we injected to the argument), we may now attempt to evaluate each term in the right-hand side of~\eqref{eq:Ex-random-chain-rule} by switching the order of these two expectations.

In a typical application of this scheme, the random vector $X$ is a uniformly random element of some $\mathcal{X} \subseteq A^I$, so that $H(X) = \log |\mathcal{X}|$, and the entropy $X_i^{\prec,x}$ is bounded from above by the logarithm of the size of the set $A_i^{\prec,x}$ of all possible values that this variable can assume, that is,
\[
  A_i^{\prec,x} \coloneqq \{y_i : y \in \mathcal{X} \text{ and } y_{\prec i} = x_{\prec i}\}.
\]
This leads to the following upper bound on the size of $\mathcal{X}$, first stated in this general form as~\cite[Lemma~4]{PalPal22}.

\begin{lemma}
  \label{lemma:randomised-chain-rule}
  Let $A$ and $I$ be finite sets and let $\prec$ be a random ordering of $I$.
  For every nonempty $\mathcal{X} \subseteq A^I$,
  \[
    \log |\mathcal{X}| \le \sum_{i \in I} \Ex\left[ \log |A_i^{\prec,x}|\right],
  \]
  where the expectation averages over both $\prec$ and a uniformly random $x \in \mathcal{X}$.
\end{lemma}

As an illustration, we prove the following result, which is implicit in the work of Linial and Luria~\cite{LinLur13}.

\begin{theorem}[\cite{LinLur13,Lur}]
  \label{thm:number-of-hypergraph-PM}
  Suppose that $\mathcal{H}$ is a $k$-uniform, $d$-regular, $n$-vertex linear hypergraph.
  The number $M(\mathcal{H})$ of perfect matchings in $\mathcal{H}$ satisfies, as $d^{1/(k-1)} \to \infty$,
  \[
    \log M(\mathcal{H}) \le \frac{n}{k} \cdot \bigl(\log d - (k-1) \cdot (1-o(1))\bigr).
  \]
\end{theorem}

Since a Steiner triple system on $\br{n}$ is nothing else than a perfect matching in the (linear) $3$-uniform hypergraph with vertex set $E(K_n)$ whose hyperedges are the edge sets of triangles of $K_n$, \cref{thm:number-of-hypergraph-PM} implies that there are at most
\[
  \exp\left(\frac{1}{3} \binom{n}{2} \cdot \bigl(\log (n-2) - 2 + o(1)\bigr)\right) = \left((1+o(1)) \cdot \frac{n}{e^2}\right)^{n^2/6}
\]
such Steiner triple systems; this estimate is the main result of~\cite{LinLur13}.
Our presentation here will closely follow Luria~\cite{Lur}, who proved \cref{thm:number-of-hypergraph-PM} under the weaker assumption that $\Delta_2(\mathcal{H}) = o(d)$.
We work with the stronger assumption that $\Delta_2(\mathcal{H}) = 1$ in order to avoid a technical complication in the final part of Luria's argument.

\begin{proof}[Proof of~\cref{thm:number-of-hypergraph-PM}]
  Denote the vertex and the edge sets of $\mathcal{H}$ by $V$ and $E$, respectively.
  Let $\mathcal{X}$ denote the set of perfect matchings in $\mathcal{H}$, each of which can be naturally viewed as a vector in $E^V$ (whose $v$-coordinate is the unique edge of the matching that contains $v$).
  Fix an ordering $\prec$ of $V$, a perfect matching $x$, and a vertex $v \in V$.
  Clearly, $A_v^{\prec,x}$ always contains the edge $x_v$.
  The key observation, however, is that an edge $f \neq x_v$ can belong to the set $A_v^{\prec,x}$ only if $v$ belongs to $f$ and $f$ is disjoint from $\bigcup_{w \prec v} x_w$.
  A moment of thought reveals that the latter holds if and only if $v$ is the $\prec$-smallest element of $\bigcup_{w \in f} x_w$.

  Let $\prec$ be a uniformly chosen random ordering of $V$.
  An elegant idea of Linial and Luria~\cite{LinLur13} is to generate $\prec$ by first assigning to each $v \in V$ a uniformly random  $\tau_v \in [0,1]$ and order the vertices according to the resulting function $\tau \colon V \to [0,1]$ (which is injective with probability one) by letting $v \prec w$ whenever $\tau_v > \tau_w$.
  This enables us to apply Jensen's inequality in the following way.
  Letting $Z_v^{\prec,x}$ be the indicator of the event $v \preceq x_v$, we have
  \begin{equation}
    \label{eq:Ex-Avprecx}
    \Ex[\log|A_v^{\prec,x}|] = \Ex[\Ex[\log |A_v^{\prec,x}| \mid \tau_v, x, Z_v^{\prec,x}]] \le \Ex[\log \Ex[|A_v^{\prec,x}| \mid \tau_v, x, Z_v^{\prec,x}]].
  \end{equation}
  As it is easy to check that, for all $v \in V$ and $W \subseteq V$,
  \[
    \Pr(v \preceq W \mid \tau_v) = \tau_v^{|W \setminus \{v\}|},
  \]
  we have, for every $f \in E \setminus \{x_v\}$ that contains $v$,
  \[
    \Pr(v \preceq x_v \mid \tau_v, x) = \tau_v^{k-1}
    \qquad
    \text{and}
    \qquad
    \Pr(v \preceq \bigcup_{w \in f} x_w \mid \tau_v, x) = \tau_v^{|\{x_w : w \in f\}| \cdot k - 1} = \tau_v^{k^2-1},
  \]
  where the final equality holds thanks to our assumption that $\mathcal{H}$ is linear (and thus every edge $f$ that is not in the matching $x$ intersects $k$ edges of $x$).
  We may conclude that
  \[
    \Ex[|A_v^{\prec,x}| \mid \tau_v, x, Z_v^{\prec,x}] \le 1 + Z_v^{\prec,x} \cdot  d \tau_v^{k(k-1)}.
  \]
  Substituting this estimate into~\eqref{eq:Ex-Avprecx}, we obtain
  \[
    \Ex[\log |A_v^{\prec,x}|]
    \le \Ex[\tau_v^{k-1} \cdot \log(1 + d \tau_v^{k(k-1)})] = \int_0^1 t^{k-1} \log(1+d t^{k(k-1)})\,dt
    = \frac{1}{k} \cdot \int_0^1 \log(1+d u^{k-1})\,du.
  \]
  Finally, letting $\delta \coloneqq d^{-1/(k-1)}$, we have
  \begin{multline*}
    \int_0^1 \log(1+du^{k-1}) \, du = \log d + \int_0^1 \log(\delta^{k-1}+u^{k-1}) \, du \le \log d + (k-1) \int_0^1 \log(u + \delta) \, du \\
    = \log d + (k-1) \cdot \bigl((1+\delta)\log(1+\delta)-\delta \log \delta - 1\bigr).
  \end{multline*}
  Since $(1+\delta)\log(1+\delta)-\delta\log \delta \to 0$ as $\delta \to 0$, the claimed upper bound on $M(\mathcal{H})$ follows from~\cref{lemma:randomised-chain-rule}.
\end{proof}

\section{Shearer's inequality.}
\label{sec:Shearer-inequality}

Consider an arbitrary vector $X$ of discrete random variables whose coordinates are indexed by the elements of a finite set $V$.
The subadditivity property of entropy (\cref{fact:entropy-subadditivity}) implies that, for every partition $\mathcal{P}$ of $V$, we have
\[
  H(X) \le \sum_{W \in \mathcal{P}} H(X_W),
\]
where we write $X_W \coloneqq (X_v)_{v \in W}$ for the projection of $X$ onto the coordinates in $W$.
The following remarkable generalisation of this inequality was proved in 1978 by Shearer, although it first appeared in print only several years later~\cite{ChuGraFraShe86}.

\begin{lemma}[Shearer's inequality]
  \label{lemma:Shearer}
  Suppose that $\mathcal{F}$ is a hypergraph on a finite set $V$ with minimum degree at least $t$.
  Then, for every discrete random vector $X$ whose coordinates are indexed by $V$, we have
  \[
    H(X) \le \frac{1}{t} \sum_{W \in \mathcal{F}} H(X_W).
  \]
\end{lemma}

We remark here that the special case $\mathcal{F} = \{V \setminus \{v\} : v \in V\}$ and $t = |V|-1$, which generalises the well-known Loomis--Whitney inequality~\cite{LooWhi49}, had been independently obtained by Te Sun Han~\cite{Han78}.
The following `book' proof of Shearer's inequality was discovered by Llewellyn and Radhakrishnan.

\begin{proof}[Proof of~\cref{lemma:Shearer}]
  Let $\prec$ be an arbitrary ordering of the elements of $V$.
  By \cref{fact:chain-rule,fact:monotonicity-of-conditional-entropy}, for every $W \in \mathcal{F}$, we have
  \[
    H(X_W) = \sum_{v \in W} H(X_v \mid (X_w : w \prec v, w \in W)) \ge \sum_{v \in W} H(X_v \mid X_{\prec v}).
  \]
  Summing the above inequality over all $W \in \mathcal{F}$ yields
  \[
    \sum_{W \in \mathcal{F}} H(X_W) \ge \sum_{v \in V} \deg_{\mathcal{F}}v \cdot H(X_v \mid X_{\prec v}) \ge t \cdot \sum_{v \in V} H(X_v \mid X_{\prec v}) = t \cdot H(X),
  \]
  where the second inequality holds as entropy is nonnegative and the equality is the chain rule (\cref{fact:chain-rule}).
\end{proof}

The following, purely combinatorial, corollary of the general lemma is already sufficient for many applications.

\begin{corollary}
  \label{cor:combinatorial-Shearer}
  Suppose that $\mathcal{F}$ is a hypergraph on a finite set $V$ with minimum degree at least $t$.
  For every finite $A$ and all $\mathcal{X} \subseteq A^V$,
  \[
    |\mathcal{X}|^t \le \prod_{W \in \mathcal{F}} |\mathcal{X}_W|,
  \]
  where $\mathcal{X}_W$ denotes the projection of $\mathcal{X}$ onto the coordinates in $W$.
\end{corollary}
\begin{proof}
  Let $X$ be a uniformly chosen random vector in $\mathcal{X}$.
  By \cref{fact:entropy-uniform,lemma:Shearer},
  \[
    \log|\mathcal{X}| = H(X) \le \frac{1}{t} \cdot \sum_{W \in \mathcal{F}} H(X_W) \le \frac{1}{t} \cdot \sum_{W \in \mathcal{F}} \log |\mathcal{X}_W|,
  \]
  as claimed.
\end{proof}

It is worth mentioning that both \cref{lemma:Shearer,cor:combinatorial-Shearer} remain true when we allow the edges of the hypergraph $\mathcal{F}$ to have nonnegative integer multiplicities (in which case, degrees count edges with their multiplicities); the proofs remain unchanged.
A further generalisation of~\cref{cor:combinatorial-Shearer} to weighted hypergraphs was considered by Friedgut~\cite{Fri04}.

One particularly elegant application of \cref{cor:combinatorial-Shearer}, due to Friedgut and Kahn~\cite{FriKah98}, gives an upper bound on the number of copies of a uniform hypergraph $F$ in another hypergraph with a given number of edges.
Their bound generalises the earlier work of Alon~\cite{Alo81}, who obtained the same bound in the case where $F$ is a graph, and is best-possible up to a multiplicative constant that depends only on $F$.
The method of~\cite{FriKah98} was later used in~\cite{JanOleRuc04} to prove strong bounds on the number of copies of a given graph in another graph with given numbers of edges and vertices (see also~\cite{HarMouSam22}).
Finally, we refer the interested reader to the recent~\cite{CohHarMouSam} for a streamlined version of the proof of the main results of~\cite{FriKah98,JanOleRuc04} that avoids the use of Shearer's inequality and linear programming duality.

As an illustration of~\cref{cor:combinatorial-Shearer}, we present here the argument of~\cite{FriKah98} in one special case, obtaining an `asymptotic' version of the famous Kruskal--Katona Theorem.
The \emph{shadow} of a family $\mathcal{H}$ of sets is the family
\[
  \partial \mathcal{H} \coloneqq \{E \setminus \{v\} : v \in E \in \mathcal{H}\}.
\]

\begin{proposition}
  \label{prop:Kruskal-Katona-weak}
  Suppose that $\mathcal{H}$ is a family of $k$-element sets.
  If $|\mathcal{H}| \ge x^k/k!$ for some real $x > 0$, then $|\partial \mathcal{H}| \ge x^{k-1}/(k-1)!$.
\end{proposition}
\begin{proof}
  Let $A \coloneqq \bigcup \mathcal{H}$ and let $\mathcal{X} \subseteq A^{\br{k}}$ be the family of all ordered sets from $\mathcal{H}$, so that $|\mathcal{X}| = k! \cdot |\mathcal{H}| \ge x^k$.
  Note that, for every $i \in \br{k}$, the projection $\mathcal{X}_{\br{k} \setminus \{i\}}$ comprises only ordered members of $\partial \mathcal{H}$ and thus its cardinality cannot exceed $(k-1)! \cdot |\partial \mathcal{H}|$.
  Consequently, by~\cref{cor:combinatorial-Shearer},
  \[
    x^{k(k-1)} \le |\mathcal{X}|^{k-1} \le \prod_{i \in \br{k}} |\mathcal{X}_{\br{k} \setminus \{i\}}| \le ((k-1)! \cdot |\partial \mathcal{H}|)^k,
  \]
  as claimed.
\end{proof}

The recent work of Chao and Yu~\cite{ChaYu24} presents another entropy-based argument that proves a strengthening of~\cref{prop:Kruskal-Katona-weak}, known as Lovász's formulation of the Kruskal--Katona Theorem, where $x^k/k!$ and $x^{k-1}/(k-1)!$ are replaced by $\binom{x}{k}$ and $\binom{x}{k-1}$, respectively.
The same authors used entropy methods to prove sharp upper bounds on the number of rainbow triangles in a graph with specified numbers of red, green, and blue edges~\cite{ChaYu,ChaYu24}, improving a `vanilla' application of Shearer's inequality by a constant factor.

Another extremely influential application of Shearer's inequality drives the proof of the following upper bound on the number of homomorphisms from a regular bipartite graph to an arbitrary graph (with loops allowed).

\begin{theorem}[\cite{GalTet04,Kah01-HC}]
  \label{thm:Galvin-Tetali}
  For every $d$-regular, $N$-vertex, bipartite graph $G$ and every graph $F$,
  \[
    |\Hom(G,F)|^{1/N} \le |\Hom(K_{d,d},F)|^{1/(2d)}.
  \]
\end{theorem}

\Cref{thm:Galvin-Tetali} is a result of Galvin and Tetali~\cite{GalTet04}, but its proof is an adaptation of the method of Kahn~\cite{Kah01-HC}, who obtained an upper bound on the number of independent sets in a bipartite regular graph.
(The method has its roots in Kahn's earlier work with Lawrenz~\cite{KahLaw99}, which studies the number of so-called rank functions on the hypercube, which are in one-to-one correspondence with homomorphisms from the hypercube to $\mathbb{Z}$.)
In fact, Kahn's result is a special case of the above theorem, as there is a well-known bijection between independent sets of $G$ and the set $\Hom(G,\indsetgraph)$.
The strongest result proved in~\cite{GalTet04} is a generalisation of \cref{thm:Galvin-Tetali} to the setting of weighted graph homomorphisms, where the vertices of $F$ are assigned positive real weights via some $\lambda \colon V(F) \to (0, \infty) $ and the weight of a homomorphism $f \in \Hom(G,F)$ is $\prod_{v \in V(G)} \lambda_{f(v)}$;
this more general result is deduced from~\cref{thm:Galvin-Tetali} by relating the sum of weights of all homomorphisms to the number of homomorphisms to an appropriate blowup of $F$.

Various extensions and variations of the entropy-based method of proof \Cref{thm:Galvin-Tetali} were used by many authors not only to give tighter bounds on the number of homomorphisms from $G$ to $F$ under further assumptions on $G$, but also to describe the typical structure of randomly chosen such homomorphism.
Some highlights include: the works of Kahn on the typical structure of a random independent set in the hypercube~\cite{Kah01-HC} and a random homomorphism from the hypercube to $\mathbb{Z}$~\cite{Kah01-RW}, further sharpened by Galvin~\cite{Gal03}, and the number of antichains in the Boolean lattice~\cite{Kah02}; and the works of Engbers and Galvin that give a rough description of the typical structure of a random homomorphism from an arbitrary regular bipartite graph~\cite{EngGal12-bip} as well as a finer description in the case where this graph is a high-dimensional torus~\cite{EgnGal12-tori}.
Several more recent works combine entropy arguments with the method of graph containers to prove extremely precise results on the typical structure of random graph homomorphisms from the hypercube~\cite{KahPar20,LiMcKPar25}, high-dimensional tori~\cite{JenKee23}, the nearest-neighbour graph on $\mathbb{Z}^d$~\cite{PelSpi,PelSpi23}, and bipartite expanders~\cite{KruLiPar}.
Finally, we mention that the recent work~\cite{BalBolNar21} uses the method of proof of \cref{thm:Galvin-Tetali} to give sharp upper bounds on the number of independent sets in certain regular hypergraphs.

\begin{proof}[Proof of \cref{thm:Galvin-Tetali}]
  Let $f \colon V(G) \to V(F)$ be a uniformly chosen random homomorphism from $G$ to $F$, so that $H(f) = \log|\Hom(G,F)|$, and let $V(G) = V_0 \cup V_1$ be an arbitrary bipartition of $G$.
  Writing $f_W$ for the restriction of $f$ to a subset $W \subseteq V(G)$, we may use the chain rule (\cref{fact:chain-rule}) to obtain
  \begin{equation}
    \label{eq:Hf-bipartite-chain-rule}
    H(f) = H(f_{V_0}) + H(f_{V_1} \mid f_{V_0}).
  \end{equation}
  Since every vertex in $V_0$ is a neighbour of exactly $d$ vertices in $V_1$, we may use Shearer's inequality to bound the first term in the right-hand side of~\eqref{eq:Hf-bipartite-chain-rule} as follows:
  \[
    H(f_{V_0}) \le \frac{1}{d} \cdot \sum_{v \in V_1} H(f_{N_v}).
  \]
  Further, using \cref{fact:monotonicity-of-conditional-entropy,fact:conditional-entropy-subadditivity}, we bound the second term in the right-hand side of~\eqref{eq:Hf-bipartite-chain-rule} as follows:
  \[
    H(f_{V_1} \mid f_{V_0}) \le \sum_{v \in V_1} H(f_{v} \mid f_{V_0}) \le \sum_{v \in V_1} H(f_{v} \mid f_{N_v}),
  \]
  where $N_v$ denotes the neighbourhood of $v$ in $G$.
  Substituting these two upper bounds back into~\eqref{eq:Hf-bipartite-chain-rule}, we obtain the key inequality
  \[
    d \cdot H(f) \le \sum_{v \in V_1} \bigl( H(f_{N_v}) + d \cdot H(f_v \mid f_{N_v}) \bigr).
  \]
  The crux of the argument is to realise that each term in the above sum is the entropy of a random homomorphism from the complete bipartite graph $K_{d,d}$ to $F$.
  Indeed, one may define such random $g^v \in \Hom(K_{d,d}, F)$ as follows.
  Denote the two partite sets of $K_{d,d}$ by $A$ and $B$, let $\varphi^v \colon A \to N_v$ be an arbitrary bijection, and let $g_A^v = f_{\varphi^v(A)}$.
  Further, conditioned on $g_A^v$, the vector $g_B^v$ comprises $d$ independent copies of $f_v$ conditioned on $f_{N_v}$.
  Since $f$ is a homomorphism and $B$ is an independent set, the resulting function $g^v$ is also a homomorphism.
  Most importantly, by construction and~\cref{fact:conditional-entropy-subadditivity},
  \[
    H(g^v) = H(g^v_A) + H(g^v_B \mid g^v_A) = H(g^v_A) + \sum_{b \in B} H(g_b^v \mid g_A^v) = H(f_{N_v}) + d \cdot H(f_v \mid f_{N_v}).
  \]
  Since $H(g^v) \le \log |\Hom(K_{d,d}, F)|$ for each $v$, by~\cref{fact:entropy-uniform}, we obtain
  \[
    d \cdot \log|\Hom(G, F)| = d \cdot H(f) \le |V_1| \cdot \log|\Hom(K_{d,d}, F)| = N/2 \cdot \log|\Hom(K_{d,d}, F)|,
  \]
  which is equivalent to the claimed inequality.
\end{proof}

As a final illustration of the usefulness of Shearer's inequality, we present a generalisation of a beautiful entropy-based proof of the edge-isoperimetric inequality for the hypercube discovered by Boucheron, Lugosi, and Massart~\cite{BouLugMas13}.
Our presentation here closely follows~\cite{DisSam25}, where optimal edge-isoperimetric inequalities are proved for a much broader class of graphs.
Given positive integers $m$ and $n$, we denote by $K_m^n$ the Cartesian product of $n$ complete graphs $K_m$, that is, the graph with vertex set $\br{m}^n$ whose edges are pairs of vectors that differ in exactly one coordinate, so that $K_2^n$ is the $n$-dimensional hypercube graph.

\begin{proposition}
  For all integers $m \ge 2$ and $n \ge 1$ and every nonempty $A \subseteq \br{m}^n$, we have
  \[
    e_{K_m^n}(A, A^c) \ge |A| \cdot (m-1)(n - \log_m|A|).
  \]
\end{proposition}
\begin{proof}
  Let $X = (X_1, \dotsc, X_n)$ be a uniformly chosen random vertex of $A$.
  For every $v \in \br{m}^n$ and each $i \in \br{n}$, denote by $v_{(i)}$ the projection of $v$ along the $i$th coordinate, that is, $v_{(i)}=(v_1,\dotsc, v_{i-1}, v_{i+1}, \dotsc, v_n)$. Further, given an $x \in A$, let $A_i(x) \subseteq \br{m}$ denote the support of $X_i$ conditioned on $X_{(i)} = x_{(i)}$.
  Our first observation is that
  \[
    e_{K_m^n}(A,A^c) = \sum_{x \in A} \sum_{i=1}^n (m-|A_i(x)|).
  \]
  Now, denoting by $k_i$ the (random) size of $A_i(X)$, we may rewrite the above identity as
  \[
    e_{K_m^n}(A, A^c) = |A| \cdot \sum_{i=1}^n \Ex[m-k_i].
  \]
  The key idea is to define the affine function $\psi \colon [0, \log m] \to \mathbb{R}$ by $\psi(x) \coloneqq (m-1)(1-x/\log m)$ and observe that $\psi(\log k) \le m-k$ for every $k \in \br{m}$.
  Consequently,
  \[
    e_{K_m^n}(A, A^c) \ge |A| \cdot \sum_{i=1}^n \Ex[\psi(\log k_i)] = |A| \cdot \sum_{i=1}^n \psi\bigl(\Ex[\log k_i]\bigr) = |A| \cdot n \cdot \psi\left(\frac{1}{n} \sum_{i=1}^n \Ex[\log k_i]\right).
  \]
  Since $\Ex[\log k_i]$ is precisely the conditional entropy $H(X_i \mid X_{(i)})$, Shearer's inequality (or even the weaker Han's inequality) implies that
  \[
    \sum_{i=1}^n \Ex[\log k_i] = \sum_{i=1}^n H(X_i \mid X_{(i)}) = n \cdot H(X) - \sum_{i=1}^n H(X_{(i)}) \le H(X) = \log |A|.
  \]
  Since $\psi$ is decreasing, we may conclude that
  \[
    e_{K_m^n}(A, A^c) \ge |A| \cdot n \cdot \psi\bigl((\log |A|)/n\bigr) = |A| \cdot (m-1)(n - \log_m|A|),
  \]
  as claimed.
\end{proof}

\section{Random homomorphisms with large entropy.}
\label{sec:Sidorenko}

In the previous section, we showed how entropy methods can be used to prove upper bounds on the number of homomorphisms from a regular bipartite graph $G$ to an arbitrary graph $F$.
The reason why this is possible is that $\log |\Hom(G,F)|$ may be viewed as the entropy of a uniformly chosen random $f \in \Hom(G,F)$ and, as the proof of \cref{thm:Galvin-Tetali} shows, $H(f)$ can be bounded from above using Shearer's inequality (\cref{lemma:Shearer}).
In principle, we could also prove a lower bound on the number of homomorphisms by supplying a lower bound on the entropy of a uniformly random $f \in \Hom(G,F)$.
However, since the distribution of $f$ can be difficult to analyse, this approach seems intractable.
Kopparty and Rossman~\cite{KopRos11} pioneered the following striking alternative to this idea:
Since the logarithm of $|\Hom(G,F)|$ is bounded from below by the entropy of \emph{any} probability distribution on $\Hom(G,F)$ (by \cref{fact:entropy-uniform}), one may obtain lower bounds on $|\Hom(G,F)|$ by constructing distributions that still have large entropies but exhibit more independence, which makes them easier to analyse.

The basic method of Kopparty and Rossman has been adapted and generalised by multitude of subsequent works, mostly in the context of the notorious conjecture of Erdős and Simonovits~\cite{ErdSim84} and Sidorenko~\cite{Sid93,Sid91}.
Given graphs $F$ and $G$, we define
\[
  t(F, G) \coloneqq |\Hom(F,G)| \cdot v_G^{-v_F};
\]
in other words, $t(F,G)$ is the probability that a random function $f \colon V(F) \to V(G)$ is a homomorphism.
\begin{conjecture}[Sidorenko and Erdős--Simonovits]
  \label{conj:Sidorenko}
  For every bipartite graph $F$ and all $G$,
  \[
    t(F,G) \ge t(K_2, G)^{e_F}.
  \]
\end{conjecture}

Although Conjecture~\ref{conj:Sidorenko} remains open, it has been shown to be true for many special families of $F$.
What makes it relevant to us is that most of the recent progress on the conjecture has relied on generalisations of the entropy-based method of Kopparty and Rossman.
The first to explore these ideas were Li and Szegedy~\cite{LiSze}, closely followed by Kim, Lee, and Lee~\cite{KimLeeLee16}; in both of these works, entropy appears only implicitly.
Having said that, the following later works~\cite{ConKimLeeLee18,ConLee17,Sze} explicitly follow the general scheme outlined above.
Several other works used similar approaches to study a number of related problems.
In particular: Kamčev, Liebanau, and Morrison~\cite{KamLieMor23} study analogues of Conjecture~\ref{conj:Sidorenko} in the context of systems of linear equations; Lee~\cite{Lee21} gives lower bounds on $t(F,G)$ for a class of nonbipartite $F$ under the assumption that $G$ is `locally dense'; and the papers~\cite{BehMorNoe24,GrzLeeLidVol22,KraVolWei25} use similar entropy-based arguments in the closely-related context of commonality of graphs.
Finally, given a pair $F, F'$ of graphs, one can consider generalising the inequality appearing in Conjecture~\ref{conj:Sidorenko} to $t(F,G)^{1/e_F} \ge t(F',G)^{1/e_{F'}}$ and ask when it is valid for all $G$.
For example, Erdős and Simonovits conjectured that it holds when $F$ and $F'$ are both paths of odd lengths and $e_F \ge e_F'$, which was proved by Sa{\u g}lam~\cite{Sag18}; an entropy-based proof of this conjecture was given in~\cite{BleRay23}.
The case where both $F$ and $F'$ are trees is studied, using entropy techniques, in~\cite{BehCruNoeSim25}.

In order to add some substance to the above discussion, we use an adaptation of the method of Kopparty and Rossman~\cite{KopRos11} to show that \cref{conj:Sidorenko} is true for every connected bipartite graph $F$ that contains a `dominating' vertex, which was originally proved by Conlon, Fox, and Sudakov~\cite{ConFoxSud10} using the technique of dependent random choice (see the excellent survey of Fox and Sudakov~\cite{FoxSud11}).

\begin{theorem}
  \label{thm:Sidorenko-dominating-vertex}
  Suppose that $F$ is a connected bipartite graph with bipartition $V(F) = A \cup B$ such that some $\alpha \in A$ is adjacent to all of $B$.
  Then, for every nonempty graph $G$,
  \[
    t(F, G) \ge t(K_2, G)^{e_F}.
  \]
\end{theorem}
\begin{proof}
  Let $\alpha$ be a vertex in $A$ with $N_\alpha = B$ and let $\beta$ be an arbitrary vertex in $B$.
  We construct a random homomorphism $f \in \Hom(F,G)$ as follows:
  \begin{enumerate}[label=(\roman*)]
  \item
    \label{item:Sidorenko-i}
    Let $(f_\alpha,f_\beta)$ be a uniformly random pair of adjacent vertices of $G$.
  \item
    \label{item:Sidorenko-ii}
    Let $(f_b)_{b \in B \setminus \{\beta\}}$ be $|B|-1$ conditionally independent copies of $f_\beta$ given $f_\alpha$.
  \item
    \label{item:Sidorenko-iii}
    Having defined $f_B$, generate the variables $(f_a)_{a \in A \setminus \{\alpha\}}$ independently, conditional on $f_B$, so that, for each $a$, the conditional distribution of $f_a$ given $f_B$ matches the conditional distribution of $f_\alpha$ given $f_{N_a}$.
  \end{enumerate}
  By definition, $f$ is a homomorphism from $F$ to $G$.
  Crucially, for every nonempty $C \subseteq B$,
  \begin{equation}
    \label{eq:Sidorenko-H-star}
    H(f_\alpha, f_C) = H(f_\alpha) + H(f_C \mid f_\alpha) = H(f_\alpha) + |C| \cdot H(f_\beta \mid f_\alpha)
    = |C| \cdot H(f_\alpha, f_\beta) - (|C|-1) \cdot H(f_\alpha),
  \end{equation}
  where the key (second) equality is a consequence of~\ref{item:Sidorenko-ii} and \cref{fact:conditional-entropy-subadditivity}.
  Further, by~\ref{item:Sidorenko-iii} and \cref{fact:monotonicity-of-conditional-entropy}, for every $a \in A \setminus \{\alpha\}$ and all sets $W$ satisfying $N_a \subseteq W \subseteq V(F) \setminus \{a\}$,
  \begin{equation}
    \label{eq:Sidorenko-H-reverse-star}
    \begin{split}
      H(f_a \mid f_W)
      & = H(f_a \mid f_B) = H(f_\alpha \mid f_{N_a}) = H(f_\alpha, f_{N_a}) - H(f_{N_a}) \\
      & = |N_a| \cdot H(f_\alpha, f_\beta) - (|N_a|-1) \cdot H(f_\alpha) - H(f_{N_a}),
    \end{split}
  \end{equation}
  where the last equality holds by~\eqref{eq:Sidorenko-H-star}.
  Let $\prec$ be an arbitrary ordering of $A \setminus \{\alpha\}$.
  By the chain rule,
  \[
    H(f) = H(f_\alpha, f_B) + \sum_{a \in A \setminus \{\alpha\}} H(f_a \mid f_B, f_\alpha, f_{\prec a}),
  \]
  where $f_{\prec a}$ stands for $(f_{a'})_{a' \in A \setminus \{\alpha\}, a' \prec a}$.
  Using~\eqref{eq:Sidorenko-H-star}, \eqref{eq:Sidorenko-H-reverse-star}, and the identity $e_F = \sum_{a \in A} |N_a|$, we may rewrite the above identity as
  \[
    H(f) = e_F \cdot H(f_\alpha, f_\beta) - (e_F - |A|) \cdot H(f_\alpha) - \sum_{a \in A \setminus \{\alpha\}} H(f_{N_a}).
  \]
  Finally, using the facts that $H(f_\alpha, f_\beta) = \log(2e_G)$, that $H(f_w) \le \log v_G$ and $H(f_{N_w}) \le |N_w| \cdot \log v_G$ for every $w \in V(F)$, which all follow from \cref{fact:entropy-uniform}, we conclude that
  \[
    H(f) \ge e_F \cdot \log(2e_G) - (e_F - |A|) \cdot \log v_G - (e_F - |N_\alpha|) \cdot \log v_G = e_F \cdot \log(2e_G) + (v_F-2e_F) \cdot \log v_G.
  \]
  Since $\log |\Hom(F,G)| \ge H(f)$, again by \cref{fact:entropy-uniform}, we conclude that
  \[
    \log t(F,G) = \log |\Hom(F,G)| - v_F \cdot \log v_G \ge e_F \cdot \bigl(\log(2e_G) - 2\log v_G\bigr) = e_F \cdot \log t(K_2, G).
  \]
\end{proof}

\section{Entropy and independence.}
\label{sec:entropy-independence}

Suppose that $X$ and $Y$ are two discrete random variables.
Recall \Cref{fact:entropy-subadditivity}, which states that $H(X, Y) \le H(X) + H(Y)$ and that equality holds if and only if $X$ and $Y$ are independent.
What can be said about the distribution of the vector $(X, Y)$ when we assume that $H(X,Y)$ is not much smaller than $H(X) + H(Y)$?
One answer to this question is provided by the following beautiful inequality due to Pinsker~\cite{Pin64}.
Recall that the \emph{total variation distance} between two distributions $P$ and $Q$ on the same finite set $\mathcal{X}$ is the quantity
\[
  \dTV(Q, P) \coloneqq \max\{Q(A) - P(A) : A \subseteq \mathcal{X}\}.
\]

\begin{proposition}[Pinsker's inequality]
  \label{prop:Pinsker}
  For every pair $X$, $Y$ of discrete random variables,
  \[
    \dTV\bigl((X,Y), X \otimes Y\bigr) \le \sqrt{\frac{H(X)+H(Y)-H(X,Y)}{2}}.
  \]
\end{proposition}

Roughly speaking, Pinsker's inequality says that if $H(X,Y)$ is `close' to $H(X) + H(Y)$, then $X$ and $Y$ are `almost independent'.
We remark that Pinsker first proved the above inequality with a greater constant; the version stated above was obtained independently by Csiszár~\cite{Csi66}, Kemperman~\cite{Kem69}, and Kullback~\cite{Kul67}.
The heart of the proof of \cref{prop:Pinsker} is the following folklore estimate.
Given $p \in (0,1)$ and $q \in [0,1]$, define
\[
  \dKL(q,p) \coloneqq \DKL{\Ber(q)}{\Ber(p)} = q \log \frac{q}{p} + (1-q) \log \frac{1-q}{1-p}.
\]

\begin{lemma}
  \label{lemma:dKL}
  For all $p \in (0,1)$ and $q \in [0,1]$,
  \[
    \dKL(q,p) \ge \frac{(p-q)^2}{2\max\{r(1-r) : \min\{p,q\} \le r \le \max\{p,q\}\}} \ge 2(p-q)^2.
  \]
\end{lemma}

\begin{proof}[Proof of \Cref{prop:Pinsker}]
  Consider an arbitrary event $A$ and let $T$ be its indicator, so that both $T(X,Y)$ and $T(X \otimes Y)$ are Bernoulli random variables with success probabilities $q \coloneqq \Pr((X,Y) \in A)$ and $p \coloneqq \Pr(X \otimes Y \in A)$, respectively.
  By \eqref{eq:subadditivity-via-DKL} and the data processing inequality (\cref{fact:data-processing}),
  \[
    H(X) + H(Y) - H(X,Y) = \DKL{(X,Y)}{X \otimes Y} \ge \DKL{T(X,Y)}{T(X \otimes Y)} = \dKL(q,p).
  \]
  The claimed inequality follows from \cref{lemma:dKL}.
\end{proof}

Pinsker's inequality, and its analogues, has recently proved to be a powerful tool in extremal and probabilistic combinatorics.
Ellis, Friedgut, Kindler, and Yehudayoff~\cite{EllFriKinYeh16} used a version of \cref{prop:Pinsker} to obtain a structural characterisation of sets for which the Loomis--Whitney inequality~\cite{LooWhi49} is nearly tight.
Kozma, Meyerovitch, Peled, and the author~\cite{KozMeyPelSam24} used entropy methods to obtain a structural characterisation of random finite metric spaces; a version of \cref{prop:Pinsker} played a central role in the argument.
Finally, a strengthening of \cref{prop:Pinsker} lies at the heart of the solution to the lower-tail problem for subgraph counts in the binomial random graph found by Kozma and the author~\cite{KozSam23}.
Several other authors have independently explored the interplay between entropy and independence~\cite{CojKrzPerZde18,CojHah21,JaiRisKoe19,ManRag17,Mon08,RagTan12}.

As an illustration of the power of \cref{prop:Pinsker}, we present an entropy-based proof of the following stability version of the Kruskal--Katona theorem (\cref{prop:Kruskal-Katona-weak}), originally proved by Keevash~\cite{Kee08}.
We refer the interested reader to~\cite[Section~3]{KozSam23} for another simple application of \cref{prop:Pinsker} in the context of counting triangle-free graphs.

\begin{theorem}
  Suppose that $\mathcal{H}$ is a family of $k$-element sets with cardinality $x^k/k!$, for some real $x > 0$.
  If $|\partial \mathcal{H}| \le (1+\varepsilon) x^{k-1} / (k-1)!$, then there is an $\lceil x \rceil$-element subset of $V \coloneqq \bigcup \mathcal{H}$ that contains all but at most $C_k\varepsilon^{1/2} x^{k-1} / (k-1)!$ members of $\partial \mathcal{H}$, where $C_k$ is a constant that depends only on $k$.
\end{theorem}
\begin{proof}
  Let $U$ be the uniformly chosen random ordered set in $\partial \mathcal{H}$ and denote by $\pi \colon V^{k-1} \to V$ the projection on the last coordinate.
  Further, for a family $\mathcal{G}$ of subsets of $V$ and a $v \in V$, let $\partial_v \mathcal{G} \coloneqq \{E \setminus \{v\} : v \in E \in \mathcal{G}\}$ be the link of $v$ in $\mathcal{G}$, cf.\ the definition of the shadow of $\mathcal{G}$.
  Given a $\delta \in [0,1]$, define
  \[
    V_\delta \coloneqq \{ v \in V : |\partial_v\partial \mathcal{H}| \le (1-\delta)x^{k-2}/(k-2)!\}
  \]
  and observe that
  \[
    \frac{(1-\delta) x^{k-2}}{(k-2)!} \cdot |V \setminus V_\delta| \le \sum_{v \in V} |\partial_v \partial \mathcal{H}| = (k-1) \cdot |\partial \mathcal{H}| \le \frac{(1+\varepsilon)x^{k-1}}{(k-2)!},
  \]
  which means that $|V \setminus V_\delta| \le (1+\varepsilon)/(1-\delta) \cdot x$.
  In particular, if we let $R$ be a randomly chosen set of $\min\{ \lceil x \rceil, |V \setminus V_\delta|\}$ vertices of $V \setminus V_\delta$, every $v \in V \setminus V_\delta$ satisfies
  \[
    \Pr(v \in R) \ge \min\left\{1, \frac{x}{|V \setminus V_\delta|}\right\} \ge \frac{1-\delta}{1+\varepsilon} \ge 1 - \delta - \varepsilon.
  \]
  Consequently,
  \[
    \Pr(U \nsubseteq  R)
    \le \Pr(U \cap V_\delta \neq \emptyset) + \Pr(U \nsubseteq R \mid U \subseteq V \setminus V_\delta)
    \le (k-1) \cdot \Pr(\pi(U) \in V_\delta) + (k-1) \cdot (\delta + \varepsilon).
  \]
  It thus suffices to show that $\Pr(\pi(U) \in V_\delta) = O(\varepsilon^{1/2})$ for some $\delta = O(\varepsilon^{1/2})$.

  Let $X = (X_1, \dotsc, X_k)$ be a uniformly chosen random ordered set of $\mathcal{H}$, so that $H(X) = \log(k! \cdot |\mathcal{H}|) = k\log x$.
  Writing $X_{(i)}$ for the projection of $X$ along the $i$th coordinate, Shearer's inequality (\cref{lemma:Shearer}) implies that
  \[
    k(k-1) \cdot \log x = (k-1) \cdot H(X) \le \sum_{i=1}^k H(X_{(i)}) = k \cdot H(X_{(1)}),
  \]
  which, by \eqref{eq:entropy-minus-relative-entropy}, means that
  \[
    \DKL{X_{(1)}}{U} = \log\bigl((k-1)! \cdot |\partial \mathcal{H}|\bigr) - H(X_{(1)}) \le \log(1+\varepsilon) \le \varepsilon.
  \]
  Further, by the data processing inequality (\cref{fact:data-processing}), we have
  \[
    \DKL{X_1}{\pi(U)} = \DKL{X_k}{\pi(U)} = \DKL{\pi(X_{(1)})}{\pi(U)} \le \DKL{X_{(1)}}{U} \le \varepsilon.
  \]
  We conclude that
  \begin{equation}
    \label{eq:KK-stability-DKL-X-from-U}
    \DKL{X_1 \otimes X_{(1)}}{\pi(U) \otimes U} = \DKL{X_1}{\pi(U)} + \DKL{X_{(1)}}{U} \le 2\varepsilon.
  \end{equation}
  We will now bound the left-hand side of~\eqref{eq:KK-stability-DKL-X-from-U} from below.

  To this end, note first that
  \[
    \sum_{i=1}^k H(X_{(i)}) - (k-1) \cdot H(X) \le k \cdot \log\bigl((k-1)! \cdot |\partial \mathcal{H}|\bigr) - k(k-1)\log x \le k\log(1+\varepsilon),
  \]
  where the first inequality follows from \cref{fact:entropy-uniform}.
  On the other hand, by the chain rule (\cref{fact:chain-rule}),
  \begin{equation}
    \label{eq:KK-stability-HX-chain-rule}
    (k-1) \cdot H(X) = \sum_{i=1}^k \sum_{j \neq i} H(X_j \mid X_{< j}),
  \end{equation}
  while the chain rule plus monotonicity of conditional entropy (\cref{fact:monotonicity-of-conditional-entropy}) yield
  \begin{equation}
    \label{eq:KK-stability-sum-HX-chain-rule}    
    \sum_{i=1}^k H(X_{(i)}) \ge H(X_{(1)}) + \sum_{i=2}^k\sum_{j\neq i} H(X_j \mid X_{< j}).
  \end{equation}
  Combining \eqref{eq:KK-stability-HX-chain-rule} and \eqref{eq:KK-stability-sum-HX-chain-rule} and using the chain rule once more, we obtain
  \[
    \sum_{i=1}^k H(X_{(i)}) - (k-1) \cdot H(X) \ge H(X_{(1)}) - \sum_{j=2}^k H(X_j \mid X_{< j}) = H(X_{(1)}) + H(X_1) - H(X).
  \]
  Recalling identity~\cref{eq:subadditivity-via-DKL}, we conclude that
  \begin{equation}
    \label{eq:KK-stability-DKL-X-from-X}
    \DKL{X}{X_1 \otimes X_{(1)}} = H(X_{(1)}) + H(X_1) - H(X) \le k \log(1+\varepsilon).
  \end{equation}
  
  Finally, let $T \colon V^k \to \{0,1\}$ be the indicator of (ordered sets of) $\mathcal{H}$ and consider the two probabilities
  \[
    q_X \coloneqq \Pr\bigl(T(X_1 \otimes X_{(1)}) = 0\bigr)
    \qquad
    \text{and}
    \qquad
    q_U \coloneqq \Pr\bigl(T(\pi(U) \otimes U) = 0\bigr).
  \]
  By the data processing inequality (\cref{fact:data-processing}), by~\eqref{eq:KK-stability-DKL-X-from-X}, and since $T(X) = 1$ with probability one,
  \[
    - \log (1-q_X) = \dKL(1, 1-q_X) = \DKL{T(X)}{T(X_1 \otimes X_{(1)})} \le \DKL{X}{X_1 \otimes X_{(1)}} \le k \log(1+\varepsilon),
  \]
  which yields $q_X \le 1 - (1+\varepsilon)^{-k} \le k\varepsilon$.
  We claim that $q_U \le 3k\varepsilon$.
  Indeed, either $q_U \le q_X \le k \varepsilon$ or $q_U > q_X$.
  In the latter case, by the data processing inequality (\cref{fact:data-processing}),
  \[
    \DKL{X_1 \otimes X_{(1)}}{\pi(U) \otimes U} \ge \DKL{T(X_1 \otimes X_{(1)})}{T(\pi(U) \otimes U)} = \dKL(q_X, q_U),
  \]
  and thus, by \cref{lemma:dKL} and~\eqref{eq:KK-stability-DKL-X-from-U},
  \[
    \frac{(q_U - k \varepsilon)^2}{2q_U} \le \frac{(q_U - q_X)^2}{2q_U} \le \dKL(q_X, q_U) \le 2 \varepsilon,
  \]
  which implies that $q_U \le 3k\varepsilon$.

  Finally, for every $v \in V$,
  \[
    \Pr\bigl(T(v \otimes U) = 0\bigr) = \frac{|\partial \mathcal{H}| - |\partial_v \mathcal{H}|}{|\partial \mathcal{H}|} \ge 1 - \frac{(k-1)!}{x^{k-1}} \cdot |\partial_v \mathcal{H}|,
  \]
  where the last inequality holds as $|\partial \mathcal{H}| \ge x^{k-1}/(k-1)!$, by the Kruskal--Katona Theorem (\cref{prop:Kruskal-Katona-weak}).
  Since $|\partial_v \mathcal{H}| \le (1-\delta) x^{k-1}/(k-1)!$ for each $v \in V_\delta$, again by \cref{prop:Kruskal-Katona-weak} and the fact that $\partial \partial_v\mathcal{H} = \partial_v \partial \mathcal{H}$ for every $v \in V$, we have
  \[
    3k \varepsilon \ge q_U \ge \delta \cdot \Pr(\pi(U) \in V_\delta).
  \]
  We may now complete the proof by taking $\delta \coloneqq \sqrt{3k\varepsilon}$.
\end{proof}

\section{The union-closed sets conjecture.}
\label{sec:union-closed}

The so-called union-closed sets conjecture, posed by Peter Frankl in 1979, is one of the most notorious conjectures in extremal set theory.
It states that, for every nonempty finite family $\mathcal{F}$ of sets that is closed under union (that is, $A \cup B \in \mathcal{F}$ for every pair $A, B \in \mathcal{F}$), there is an element $x \in \bigcup \mathcal{F}$ that belongs to at least half of all the sets in $\mathcal{F}$.
Despite its innocuous formulation, the conjecture remains open to this day.
Relatively little progress towards its solution had been made until Justin Gilmer~\cite{Gil} found a beautiful and surprising entropy-based argument showing that every union-closed family $\mathcal{F}$ admits an element $x \in \bigcup \mathcal{F}$ that belongs to at least one percent of all sets in $\mathcal{F}$.
Gilmer predicted that the key technical lemma in his argument could be improved to yield a stronger version of his result with one percent replaced by $(3-\sqrt{5})/2 \approx 38 \%$.
Soon afterwards, this feat was independently achieved by four groups of authors~\cite{AlwHuaSel24,ChaLov,Peb,Saw}.
Finally, we mention that two subsequent works~\cite{Cam,Yu23} used a modification of Gilmer's approach (already suggested by Sawin~\cite{Saw}) to obtain a slightly improved bound on the maximum frequency of an element in a union-closed family.

In the remainder of this section, we present a sketch of Gilmer's lovely argument, leaving out a key technical inequality whose proof requires some nontrivial calculus.
The heart of the matter is the following beautiful theorem.

\begin{theorem}
  \label{thm:Gilmer}
  Let $A$ and $B$ be two independent samples from a distribution over subsets of~$\br{n}$.
  If $H(A) > 0$ and $\max_i \Pr(i \in A) < (3-\sqrt{5})/2$, then $H(A \cup B) > H(A)$.
\end{theorem}

Before we sketch the proof of~\cref{thm:Gilmer}, we explain how it implies the statement about union-closed families.
Let $\mathcal{F}$ be a finite collection of sets that is closed under union and let $A$ and $B$ be two independent uniformly random members of $\mathcal{F}$.
If no element $x \in \bigcup \mathcal{F}$ belonged to at least $(3-\sqrt{5})/2$-proportion of all sets in $\mathcal{F}$, then \cref{thm:Gilmer} would imply that $H(A \cup B) > H(A) = \log |\mathcal{F}|$.
However, since $\mathcal{F}$ is union-closed, $A \cup B$ is an element of $\mathcal{F}$ and thus $H(A \cup B) \le \log |\mathcal{F}|$ by \cref{fact:entropy-uniform}, a contradiction.

\begin{proof}[Proof of \cref{thm:Gilmer} (sketch)]
  Let $X, Y \in \{0,1\}^n$ be the characteristic vectors of $A$ and $B$, respectively.
  Write $\vee$ for the coordinate-wise maximum operator.
  By the chain rule (\cref{fact:chain-rule}),
  \[
    H(A \cup B) - H(A) = H(X \vee Y) - H(X) = \sum_{i=1}^n \bigl(H(X_i \vee Y_i \mid X_{< i} \vee Y_{< i}) - H(X_i \mid X_{< i})\bigr).
  \]
  Fix some $i \in \br{n}$.
  Since conditioning on more information only reduces entropy (\cref{fact:monotonicity-of-conditional-entropy}),
  \[
    H(X_i \vee Y_i \mid X_{< i} \vee Y_{< i}) - H(X_i \mid X_{< i}) \ge H(X_i \vee Y_i \mid X_{< i}, Y_{< i}) - H(X_i \mid X_{< i}) \eqqcolon \Delta_i.
  \]
  Writing $p_i \coloneqq \Ex[X_i \mid X_{<i}]$ and $q_i \coloneqq \Ex[Y_i \mid Y_{< i}]$ and recalling the definition of conditional entropy, one quickly realises that the assumption that $X$ and $Y$ are independent means that
  \begin{equation}
    \label{eq:union-closed-Delta-i}    
    \Delta_i = \Ex[h(p_i + q_i - p_iq_i) - h(p_i)],
  \end{equation}
  where $h$ is the binary entropy function defined in \Cref{sec:binary-entropy}.
  It clearly suffices to argue that $\Delta_i \ge 0$ for every $i$ and that $\Delta_i > 0$ unless $p_i = q_i = 0$ with probability one.
  It turns out that this follows from our main assumption that $\Ex[p_i] = \Ex[q_i] = \Pr(i \in A) < (3-\sqrt{5})/2$.
  If both $p_i$ and $q_i$ were constant (which is always true when $i=1$), verifying that $\Delta_i \ge 0$ (and that $\Delta_i > 0$ unless $p_i = q_i = 0$) would be a fairly routine calculus exercise.
  In the case where $p_i$ and $q_i$ are random, this requires substantial technical work, see~\cite{AlwHuaSel24,ChaLov,Peb,Saw}.
  We just remark that the significance of $(3-\sqrt{5})/2$ is that it is the only nontrivial solution to the equation $h(2p-p^2) = h(p)$ and that the heart of the matter seems to lie in proving that $h(p^2) \ge (\sqrt{5}+1)/2 \cdot ph(p)$ for all $p \in [0,1]$, which had been proved by Boppana in the 1980s, see~\cite{Bop}.
\end{proof}

\section{Entropy-based approach to the Turán problem.}
\label{sec:Turan-via-entropy}

The recent work of Chao and Yu~\cite{ChaYu-Turan} makes an interesting connection between entropy and the classical Turán problem for hypergraphs that has already been used to prove new bounds on Turán densities of several families of hypergraphs~\cite{IlkYan,Liu-Mantel,Liu-spectral,MaZhu}.
Given a family $\mathcal{F}$ of $k$-uniform hypergraphs, we say that a hypergraph $G$ is \emph{$\mathcal{F}$-free} if $G$ does not contain any member of $\mathcal{F}$ as a subgraph.
Define the \emph{Turán number} $\ex(n, \mathcal{F})$ as the largest number of edges in an $\mathcal{F}$-free hypergraph with $n$ vertices and let
\[
  \pi(\mathcal{F}) \coloneqq \lim_{n \to \infty} \ex(n, \mathcal{F}) \cdot \binom{n}{k}^{-1}.
\]
The \emph{blowup density} of a $k$-uniform hypergraph $G$ is the quantity
\[
  b(G) \coloneqq k! \cdot \sup\left\{\sum_{A \in G} \prod_{v \in A} x_v : x \in [0,1]^{V(G)}, \, \sum_{v \in V(G)} x_v = 1\right\},
\]
which is naturally interpreted as the largest edge density of a blowup of $G$.
A hypergraph $G$ is called \emph{$\mathcal{F}$-hom-free} if $G$ does not contain a homomorphic copy of any member of $\mathcal{F}$, that is, if $\Hom(F,G) = \emptyset$ for every $F \in \mathcal{F}$.
It is well-known that, for every family $\mathcal{F}$,
\begin{equation}
  \label{eq:pi-via-b}
  \pi(\mathcal{F}) = \sup \{ b(G) : \text{$G$ is $\mathcal{F}$-hom-free}\}  = \sup\left\{k! \cdot e_G \cdot v_G^{-k} : \text{$G$ is $\mathcal{F}$-hom-free} \right\}.
\end{equation}
Chao and Yu~\cite{ChaYu-Turan} provide the following alternative description of $\pi(\mathcal{F})$.

\begin{proposition}[\cite{ChaYu-Turan}]
  \label{prop:Turan-density-entropy}
  For any family $\mathcal{F}$ of $k$-uniform hypergraphs, $\log \pi(\mathcal{F})$ is the supremum of $H(X_1, \dotsc, X_k) - kH(X_1)$ over uniform orderings $(X_1, \dotsc, X_k)$ of a random edge of some nonempty $\mathcal{F}$-hom-free $k$-uniform hypergraph $G$.
\end{proposition}
\begin{proof}
  Denote the supremum by $S$.
  Given an arbitrary nonempty $\mathcal{F}$-hom-free hypergraph $G$, we may let $(X_1, \dotsc, X_k)$ be the uniformly chosen random ordered edge of $G$ to obtain
  \[
    S \ge H(X_1, \dotsc, X_k) - kH(X_1) = \log(k! \cdot e_G) - kH(X_1) \ge \log(k! \cdot e_G) - k\log v_G.
  \]
  Taking the supremum over all $\mathcal{F}$-hom-free $G$ yields $S \ge \log \pi(\mathcal{F})$.
  For the reverse inequality, suppose that $(X_1, \dotsc, X_k)$ is the uniform ordering of a random edge of some $\mathcal{F}$-hom-free hypergraph $G$ and let $T \colon V(G)^k \to \{0,1\}$ be the indicator of (ordered) edges of $G$.
  Write $X_1^{\otimes k}$ for the $k$-dimensional vector $X_1 \otimes \dotsb \otimes X_1$.
  By \cref{fact:entropy-subadditivity}, identity \cref{eq:subadditivity-via-DKL}, and the data processing inequality (\cref{fact:data-processing}),
  \begin{multline*}
    kH(X_1) - H(X_1, \dotsc, X_k)
    = H(X_1^{\otimes k}) - H(X_1, \dotsc, X_k)
    = \DKL{(X_1, \dotsc, X_k)}{X_1^{\otimes k}} \\
    \ge \DKL{T(X_1, \dotsc, X_k)}{T(X_1^{\otimes k})}
    = \dKL(1, \Pr(T(X_1^{\otimes k})=1))
    = - \log \Pr(T(X_1^{\otimes k}) = 1).
  \end{multline*}
  Finally, a moment's thought reveals that the supremum of $\Pr(T(X_1^{\otimes k})=1)$ over all random $X_1 \in V(G)$ is nothing else but $b(G)$, which implies that $S \le \log b(G) \le \log \pi(\mathcal{F})$.
\end{proof}

To illustrate~\cref{prop:Turan-density-entropy}, Chao and Yu~\cite{ChaYu-Turan} provided an entropy-based proof of Turán's theorem that is based on the ideas described in \Cref{sec:Sidorenko}, which we present in the remainder of this section.
A~key role is played by the following lemma.

\begin{lemma}[\cite{ChaYu-Turan}]
  \label{lemma:disjoint-supp}
  Let $X_1, \dotsc, X_n$ be $\mathcal{X}$-valued random variables and let
  \[
    s \coloneqq \max_{x \in \mathcal{X}} |\{i \in \br{n} : \Pr(X_i = x) > 0\}|.
  \]
  There exists a random variable $I \in \br{n}$ that is independent of $X_1, \dotsc, X_n$ and satisfies
  \[
    s \cdot \exp(H(X_I)) \ge \sum_{i=1}^n \exp(H(X_i)).
  \]
\end{lemma}
\begin{proof}
  Denote by $\Sigma$ the sum appearing in the statement of the lemma and, for each $i \in \br{n}$, let $p_i \coloneqq \exp(H(X_i)) / \Sigma$.
  Let $I \in \br{n}$ be the random index satisfying $\Pr(I = i) = p_i$ for all $i$.
  It follows from the chain rule (\cref{fact:chain-rule}) that
  \[
    H(X_I) + H(I \mid X_I) = H(X_I, I) = H(I) + H(X_I \mid I) = \sum_{i=1}^{n} p_i \cdot \left( -\log p_i + H(X_i) \right) = \log \Sigma.
  \]
  Further, since knowing $X_I$ leaves at most $s$ options for $I$, we have $H(I \mid X_I) \le \log s$.
  Substituting this inequality into the above identity yields the assertion of the lemma.
\end{proof}

Fix an integer $r \ge 2$.
The inequality $\pi(K_{r+1}) \ge 1-1/r$ is easy to prove.
For example, one may deduce it from \cref{prop:Turan-density-entropy} by considering a uniformly random ordered edge $(X_1, X_2)$ of $K_r$, which satisfies $H(X_1, X_2) - 2H(X_1) = \log(r(r-1)) - 2\log r = \log(1-1/r)$.
For the reverse inequality, suppose that $(X_1, X_2)$ is a uniformly ordered random edge of some $K_{r+1}$-free graph $G$.
Our goal is to prove that
\[
  q\coloneqq \exp(H(X_1,X_2)-2H(X_1)) \le 1-1/r.
\]
The heart of the argument is the following statement.

\begin{claim}
  For every positive integer $N$, there exist random vectors $T_1, \dotsc, T_N \in V(G)^N$ such that, for all $i, j \in \br{N}$,
  \begin{enumerate}[label=(\roman*)]
  \item
    \label{item:Turan-stars}
    $T_{i,i}$ is adjacent to each of $T_{i,1}, \dotsc, T_{i,i-1}$,
  \item
    \label{item:Turan-entropy}
    $H(T_i) = (N-i+1) \cdot H(X_1) + (i-1) \cdot H(X_2 \mid X_1) = N \cdot H(X_1) + (i-1) \cdot \log q$, and
  \item
    \label{item:Turan-marginals}
    $T_{i,j}$ has the same distribution as $X_1$.
  \end{enumerate}
\end{claim}
\begin{proof}[Proof (sketch)]
  For each $i \in \br{N}$, we construct $T_i$ as follows (cf.~the proof of~\cref{thm:Sidorenko-dominating-vertex}).
  First, let $T_{i,i}, \dotsc, T_{i,N}$ be independent random  samples from the distribution of $X_1$.
  Second, given $T_{i,i}$, let $T_{i,1}, \dotsc, T_{i,i-1}$ be conditionally independent copies of $X_2$ conditioned on $X_1 = T_{i,i}$.
\end{proof}

It follows from \ref{item:Turan-stars} that if, for some sequence $t \in V(G)^N$, we have $\Pr(T_i = t) > 0$ for every $i \in C \subseteq \br{N}$, then the vertices $(t_i)_{i \in C}$ are pairwise adjacent.
Since $G$ is $K_{r+1}$-free, no such $t$ can belong to the support of more than $r$ among the $T_1, \dotsc, T_N$.
Using \cref{lemma:disjoint-supp} and~\ref{item:Turan-entropy}, we can define a random variable $I \in \br{N}$ satisfying
\[
  r \cdot \exp(H(T_I)) \ge \sum_{i=1}^N \exp(H(T_i)) = \sum_{i=1}^N q^{i-1} \cdot \exp(N \cdot H(X_1)).
\]
On the other hand, since each of the $N$ coordinates of $T_I$ has the same distribution as $X_1$, by~\ref{item:Turan-marginals}, we have $H(T_I) \le N \cdot H(X_1)$.
Substituting this estimate into the above inequality and letting $N$ tend to infinity, we obtain the inequality $r \ge 1/(1-q)$, which is equivalent to the desired estimate $q \le 1-1/r$.


\section*{Acknowledgements.}
I would like to thank Asaf Cohen Antonir, Shira Ben Dor, Marcelo Campos, David Conlon, Sahar Diskin, Ehud Friedgut, Matan Harel, Ilay Hoshen, Vishesh Jain, Matthew Jenssen, Gady Kozma, Eden Kuperwasser, Tom Meyerovitch, Frank Mousset, Rajko Nenadov, Jinyoung Park, Ron Peled, Yinon Spinka, and Adam (Zsolt) Wagner for many interesting and enriching conversations on the topic of entropy.
These discussions have greatly influenced the content of this paper.
Having said that, all errors and omissions are entirely my own.
Finally, special thanks to David Conlon, Eden Kuperwasser, Joonkyung Lee, and Hung-Hsun Yu for their comments on an earlier version of this paper.

\bibliographystyle{abbrv}
\bibliography{ICM-entropy}
\end{document}